\documentclass[12pt]{article}
\usepackage[french,english]{babel}
\usepackage[utf8]{inputenc}
\usepackage[T1]{fontenc}
\usepackage{amsthm}
\usepackage{amssymb}
\usepackage{amsmath}
\usepackage{stmaryrd}
\usepackage{makeidx}
\usepackage{graphicx}
\usepackage{caption}
\theoremstyle{plain}
\usepackage[colorlinks=true,breaklinks=true,linkcolor=black]{hyperref} 
\usepackage{amsfonts}
\usepackage{appendix}
\usepackage{ulem}
\usepackage{comment}

\newtheorem{theorem}{Theorem}[section] 

\newtheorem{corollary}[theorem]{Corollary}
\newtheorem{definition}[theorem]{Definition}

\newtheorem{assumption}[theorem]{Assumption}
\newtheorem{lemma}[theorem]{Lemma}

\newtheorem{remark}[theorem]{Remark}
\newtheorem{definition/proposition}[theorem]{Definition/Proposition}
\usepackage{dsfont}
\date{}
\makeindex
\usepackage{mathrsfs}
\usepackage[all]{xy}
\theoremstyle{definition}

\usepackage{geometry}
\title{A new axiomatics for masures II}
\author{Auguste \textsc{Hébert} \\Université de Lorraine, Institut Élie Cartan de Lorraine, F-54000 Nancy, France\\ UMR 7502,
auguste.hebert@univ-lorraine.fr}

\theoremstyle{definition}\newtheorem{thm*}{Theorem}

\theoremstyle{definition}
\newtheorem{proposition*}[thmbis]{Proposition}

\theoremstyle{definition}
\newtheorem{proposition**}[thmter]{Proposition}

\makeatletter \@addtoreset{figure}{section}\makeatother

\newcommand{\R}{\mathbb{R}}
\newcommand{\A}{\mathbb{A}}
\newcommand{\AC}{\mathcal{A}}
\newcommand{\N}{\mathbb{N}}
\newcommand{\Z}{\mathbb{Z}}

\newcommand{\C}{\mathbb{C}}

\newcommand{\I}{\mathcal{I}}
\newcommand{\T}{\mathcal{T}}

\newcommand{\q}{\mathfrak{q}}
\newcommand{\s}{\mathfrak{s}}

\newcommand{\RR}{\mathfrak{R}}

\newcommand{\M}{\mathcal{M}}

\newcommand{\In}{\mathrm{Int}}

\newcommand{\HC}{\mathcal{H}}

\newcommand{\KC}{\mathcal{K}}

\newcommand{\XC}{\mathcal{X}}

\newcommand{\SC}{\mathcal{S}}

\newcommand{\LC}{\mathcal{L}}
\newcommand{\MC}{\mathcal{M}}

\newcommand{\efface}[1]{}

\newcommand{\LCC}{\mathscr{L}}

\usepackage[left,modulo]{lineno}

\hypersetup{
pdfauthor={Hébert, Auguste},
pdftitle={A new axiomatics for masures II},
}

\begin{document}


\maketitle


\begin{abstract}
Masures are generalizations of Bruhat-Tits buildings. They were introduced by Gaussent and Rousseau in order to study Kac-Moody groups over valued fields. We prove that the intersection of two apartments of  a masure is convex. Using this, we simplify the axiomatic definition of masures given by Rousseau.

\end{abstract}

\section{Introduction}
Bruhat-Tits buildings are an important tool in the study of reductive groups over non-Archimedean local fields. They were introduced by Bruhat and Tits in  \cite{bruhat1972groupes} and \cite{bruhat1984groupes}. Kac-Moody groups (à la Tits) are interesting infinite dimensional (if not reductive) generalizations of reductive groups.  In order to study them over fields endowed with a discrete valuation,  Gaussent and Rousseau introduced in \cite{gaussent2008kac} some spaces similar to Bruhat-Tits buildings, called  masures (also known as hovels), on which these groups act. Charignon and Rousseau generalized this construction in \cite{charignon2010immeubles},  \cite{rousseau2016groupes} and \cite{rousseau2017almost}: Charignon treated the almost split case and Rousseau suppressed restrictions on the base field and on the group. Thanks to these works, a masure is now associated with each almost split Kac-Moody group over a valued field (with some additional assumptions on the field in the non-split case, see \cite{rousseau2017almost}).
Masures enable to obtain results on the representation theory of almost split Kac-Moody groups over non-Archimedean local fields. For example, Bardy-Panse, Gaussent and Rousseau used them to associate with each such group a spherical and an Iwahori-Hecke algebra (see \cite{gaussent2014spherical} and \cite{bardy2016iwahori}, these algebras were already defined in the split affine case by Braverman, Kazhdan and Patnaik in \cite{braverman2011spherical} and \cite{braverman2016iwahori}).

Let $G$ be a split reductive group over a valued field $(\KC,\omega)$, where $\omega:\KC\rightarrow \R\cup\{\infty\}$ is a valuation. Let $T$ be a maximal split torus of $G$. Let $\A=Y\otimes \R$, where $Y$ is the cocharacter lattice of $(G,T)$. As a set, $\I=G\times \A/\sim$, where $\sim$ is some equivalence relation on $G\times \A$. The definition of $\sim$ is complicated and based on the notion of parahoric subgroups. However, many of the properties of $\I$ can be recovered from the fact that it satisfies the crucial properties (I1) and (I2) below. Let $\Phi$ denote the root system of $(G,T)$ and $\Lambda=\omega(\KC)\setminus\{\infty\}$. For $\alpha\in \Phi$ and $k\in \Lambda$, let $H_{\alpha,k}=\{x\in \A|\alpha(x)=-k\}$. The elements of the form $H_{\alpha,k}$, for $\alpha\in \Phi$ and $k\in \Lambda$ are called \textbf{walls}. A half-space of $\A$ delimited by some wall is called a half-apartment of $\A$.  An \textbf{apartment} of $\I$ (resp. a \textbf{half-apartment} of $\I$) is a set of the form $g.\A$ (resp. $g.D$), for some $g\in G$ (resp. and some half-apartment $D$ of $\A$). We call a set \textbf{enclosed} if it is a finite intersection of half-apartments.
Then $\I$ satisfies the following properties:\begin{itemize}
\item[(I1)] for  all $x,y\in \I$, there exists an apartment containing $\{x,y\}$,

\item[(I2)] if $A$ and $A'$ are two apartments of $\I$, then $A\cap A'$ is enclosed in $A$ and there exists $g\in G$ such that $g.A=A'$ and such that $g$ fixes $A\cap A'$.
\end{itemize}

 Note that (I1) is a building theoretic translation of the Cartan decomposition of $G$. 

Let $G$ be a split Kac-Moody group over a valued field $(\KC,\omega)$. Then similarly to the reductive case, the masure $\I$ of $G$ is defined as $G\times \A/\sim$, for some equivalence relation $\sim$ on $G\times \A$ (see \cite{gaussent2008kac} and \cite{rousseau2016groupes}). As  the Cartan decomposition does not hold in $G$ (unless $G$ is reductive), property (I1) is not necessarily satisfied by $\I$. It can be replaced by an axiom involving ``chimneys'', which are certain objects at the infinity of $\I$ (see (MA iii) in \ref{subsubDéfinition des masures}). This axiom corresponds to the Iwasawa and the Birkhoff decompositions in $G$. 
Gaussent and Rousseau proved  weak versions of (I2) in \cite{gaussent2008kac} and \cite{rousseau2011masures}. More precisely, let $\T$ be the Tits cone of $\A$. If $x,y\in \I$, one writes $x\leq y$ if there exists $g\in G$ such that $g.x,g.y\in \A$ and  $g.y-g.x\in \T$. It is proved in \cite{gaussent2008kac} that if $A,A'$ are two apartments, if $x,y\in A\cap A'$ are such that $x\leq y$, then $[x,y]_A$ (the line segment in $A$ joining $x$ to $y$) is equal to $[x,y]_{A'}$ and there exists $g\in G$ such that $g.A=A'$ and such that $g$ fixes $[x,y]_A$. This property is called preordered convexity and is fundamental to most of the applications of the theory of masures so far.   However, very few was known about convexity properties for pairs of points which are not preordered.  In \cite{hebert2020new}, we proved that when $\I$ is associated with an affine Kac-Moody group, then $\I$ satisfies (I2). In general we proved that (I2) is satisfied for pairs of apartments $A,A'$ such that $A\cap A'$ is ``large enough'' (more precisely, when $A\cap A'$ contains a generic ray, see \cite[Theorem 4.22]{hebert2020new}). We used this to simplify the definition of masures. In this paper we prove that (I2) is satisfied without assumption on $A\cap A'$.   More precisely, let $\I$ be a masure in the sense of  \cite[Définition 2.1]{rousseau2011masures}, satisfying some mild technical assumption (see assumption~\ref{assFreeness}). Then:

\begin{thm*}\label{thmIntro}(see Theorem~\ref{thmIntersection_apartments})
Let $A$ and $A'$ be two apartments of $\I$. Then $A\cap A'$ is enclosed and there exists an apartment isomorphism $\phi:A\rightarrow A'$ fixing $A\cap A'$.
\end{thm*}

Note that in the theorem above, the half-apartments of $\A$ are delimited by real roots of $(G,T)$ (and not by imaginary roots as it was the case in \cite{gaussent2008kac}).

 We then use this theorem to give a simplified definition of masures equivalent to \cite[Définition 2.1]{rousseau2011masures}. This also simplifies the definition given in \cite[Theorem 1.5]{hebert2020new} (see Corollary~\ref{corDefinition}).

\paragraph{Framework}

Actually we do not limit our study to masures associated with Kac-Moody groups: for us a masure is a set satisfying the axioms of \cite{rousseau2011masures} and whose apartments are associated with  a root generating system (and thus with a Kac-Moody matrix). We do not assume the existence of a group acting strongly transitively on it. Our results apply to the masures  associated with split Kac-Moody groups over valued fields constructed in \cite{gaussent2008kac} and \cite{rousseau2016groupes} and to masures associated with almost split Kac-Moody groups over valued fields (satisfying some additional conditions, see \cite[6.1]{rousseau2017almost}) in \cite{rousseau2017almost}. Contrary to \cite{rousseau2011masures}, we assume that the family of simple coroots is free in $\A$ (see assumption~\ref{assFreeness}).

\paragraph{Comments on the proof of Theorem~\ref{thmIntro}}  In \cite{hebert2020new}, we proved that if $A$ and $A'$ are apartments containing a common generic ray, then $A\cap A'$ is enclosed by   following the steps below:

\begin{itemize}
\item[(1)] We prove that if $B,B'$ are any two apartments, then $B\cap B'$ can be written as a finite union of enclosed subsets of $B$.

\item[(2)] We prove that $A\cap A'$ is convex. Using (1),  we deduce that it is enclosed.
\end{itemize}

As we already proved (1) in \cite{hebert2020new}, the main difficulty in the proof of Theorem~\ref{thmIntro} is to prove the convexity of $A\cap A'$ (without assumption on $A$ and $A'$). In this paper, we prove it directly, without using step (2) of \cite{hebert2020new}. Our proof of the convexity of $A\cap A'$ is based on the study of some kind of Hecke paths, that is, on the study of the images of line segments by  retractions centered at a sector-germ. Our proof of the convexity is actually simpler than the proof of step (2) in \cite{hebert2020new} (and more general).

\paragraph{Organization of the paper} In section~\ref{secApartment standard}, we  describe the general framework
and recall the definition of masures.

In section~\ref{secIntersection_apartments}, we prove Theorem~\ref{thmIntro}.

\paragraph{Funding}
The author was supported by the ANR grant ANR-15-CE40-0012. 

\paragraph{Acknowledgment} I would like to thank the referee for his/her valuable comments and suggestions.

\section{General framework, Masure}\label{secApartment standard}

In this section, we define our framework and recall the definition of masures. We give the definition of \cite{hebert2020new}. 

\subsection{Standard apartment}

\subsubsection{Root generating system}

Let $A$ be a \textbf{Kac-Moody matrix} (also known as generalized Cartan matrix) i.e a square matrix $A=(a_{i,j})_{i,j\in I}$ with integers coefficients, indexed by  a finite set $I$ and satisfying: 
\begin{enumerate}
\item $\forall i\in I,\ a_{i,i}=2$

\item $\forall (i,j)\in I^2|i \neq j,\ a_{i,j}\leq 0$

\item $\forall (i,j)\in I^2,\ a_{i,j}=0 \Leftrightarrow a_{j,i}=0$.
\end{enumerate}

 A  \textbf{root generating system}   of type $A$ is a $5$-tuple $\mathcal{S}=(A,X,Y,(\alpha_i)_{i\in I},(\alpha_i^\vee)_{i\in I})$ made of a Kac-Moody matrix $A$ indexed by $I$, of two dual free $\Z$-modules $X$ (of \textbf{characters}) and $Y$ (of \textbf{cocharacters}) of finite rank $\mathrm{rk}(X)$, a family $(\alpha_i)_{i\in I}$ (of \textbf{simple roots}) in $X$ and a family $(\alpha_i^\vee)_{i\in I}$ (of \textbf{simple coroots}) in $Y$. They have to satisfy the following compatibility condition: $a_{i,j}=\alpha_j(\alpha_i^\vee)$ for all $i,j\in I$.
 
 \begin{assumption}\label{assFreeness}
 We assume that   $(\alpha_i)_{i\in I}$ is free in $\A^*$ and that    $(\alpha_i^\vee)_{i\in I}$ is free in   $\A$. 
 \end{assumption}

Let $\A=Y\otimes \R$. Every element of $X$ induces a linear form on $\A$. We consider $X$ as a subset of the dual $\A^*$ of $\A$: the $\alpha_i$, $i\in I$ are viewed as linear forms on $\A$. For $i\in I$, we define an involution $r_i$ of $\A$ by $r_i(v)=v-\alpha_i(v)\alpha_i^\vee$ for all $v\in \A$. Its space of fixed points is $\ker \alpha_i$. The subgroup of $\mathrm{GL}(\A)$ generated by the $\alpha_i$ for $i\in I$ is denoted by $W^v$ and is called the \textbf{vectorial Weyl group} of $\mathcal S$. Then $(W^v,\{r_i|i\in I\})$ is a Coxeter system.

One defines an action of the group $W^v$ on $\A^*$ as follows: if $x\in \A$, $w\in W^v$ and $\alpha\in \A^*$, then $(w.\alpha)(x)=\alpha(w^{-1}.x)$. Let $\Phi=\{w.\alpha_i|(w,i)\in W^v\times I\}$ be the set of \textbf{(real) roots}. Then $\Phi\subset Q$, where $Q=\bigoplus_{i\in I}\Z\alpha_i$. Let $Q^+=\bigoplus_{i\in I} \N \alpha_i$, $\Phi^+=Q^+\cap \Phi$ and $\Phi^-=(-Q^+)\cap \Phi$. Then $\Phi=\Phi^+\sqcup \Phi^-$.

We set $Q^\vee_\R=\bigoplus_{i\in I} \R\alpha_i^\vee$ and $Q^\vee_{\R_+}=\bigoplus_{i\in I} \R_+\alpha_i^\vee$. For $x,y\in \A$, we write $x\leq_{Q^\vee_\R} y$ if $y-x\in Q^\vee_{\R_+}$.

\subsubsection{Vectorial faces and Tits cone}\label{subsubsecVectorial faces}

Define $C_f^v=\{v\in \A|\  \alpha_i(v)>0,\ \forall i\in I\}$. We call it the \textbf{fundamental vectorial chamber}. For $J\subset I$, one sets $F^v(J)=\{v\in \A|\ \alpha_i(v)=0\ \forall i\in J,\alpha_i(v)>0\ \forall i\in J\backslash I\}$. Then the closure $\overline{C_f^v}$ of $C_f^v$ is the union of the $F^v(J)$ for $J\subset I$. The \textbf{positive} (resp. \textbf{negative}) \textbf{vectorial faces} are the sets $w.F^v(J)$ (resp. $-w.F^v(J)$) for $w\in W^v$  and $J\subset I$. A \textbf{vectorial face} is either a positive vectorial face or a negative vectorial face. We call \textbf{positive chamber} (resp. \textbf{negative}) every cone  of the form $w.C_f^v$ for some $w\in W^v$ (resp. $-w.C_f^v$).  For all $x\in C_f^v$ and for all $w\in W^v$, $w.x=x$ implies that $w=1$. In particular the action of $w$ on the positive chambers is simply transitive. The \textbf{Tits cone} $\mathcal T$ is defined by $\mathcal{T}=\bigcup_{w\in W^v} w.\overline{C^v_f}$.
One defines a $W^v$-invariant preorder $\leq$  on $\A$, the \textbf{Tits preorder}   by: \[\forall (x,y)\in \A\mathrm{}^2,\ x\leq y\ \Leftrightarrow\ y-x\in \mathcal{T}.\]

\subsubsection{Affine Weyl group of  $\A$}\label{subsubGroupe de Weyl}
We now define the Weyl group $W$ of $\A$.   If $\XC$ is an affine subspace of $\A$, one denotes by $\vec{\XC}$ its direction. One equips $\A$ with a family $\M$ of affine hyperplanes called \textbf{real walls} such that: 
\begin{enumerate}

\item For all $M\in \M$, there exists $\alpha_M\in \Phi$ such that $\vec{M}=\ker(\alpha_M)$.

\item For all $\alpha\in \Phi$, there exists an infinite number of hyperplanes $M\in \M$ such that $\alpha=\alpha_M$.

\item\label{itStabilité des murs sous le groupe de weyl} If $M\in \M$, we denote by $r_M$ the reflection with respect to the hyperplane $M$ whose associated linear map is $r_{\alpha_M}$. We assume that the group $W$ generated by the $r_M$ for $M\in \M$ stabilizes $\M$. 
\end{enumerate}

The group $W$  is the Weyl group of $\A$. We assume that $0$ is special (i.e we assume that $\ker \alpha\in \MC$ for every $\alpha\in \Phi$) and thus $W\supset W^v$. 

For $\alpha\in \A^*$ and $k\in \R$, set $M(\alpha,k)=\{v\in \A| \alpha(v)+k=0\}$. Then for all $M\in \M$, there exists $\alpha\in \Phi$ and $k_M\in \R$ such that $M=M(\alpha,k_M)$. For $\alpha\in \Phi$, set $\Lambda_\alpha=\{k_M|\ M\in \mathcal{M}\mathrm{\ and\ }\vec{M}=\ker(\alpha)\}$. Then $\Lambda_{w.\alpha}=\Lambda_\alpha$ for all $w\in W^v$ and $\alpha\in \Phi$. 

If $\alpha\in \Phi$, one denotes by  $\tilde{\Lambda}_\alpha$ the subgroup of $\R$ generated by $\Lambda_\alpha$.  By~(\ref{itStabilité des murs sous le groupe de weyl}), $\Lambda_\alpha=\Lambda_{\alpha}+2\tilde{\Lambda}_\alpha$ for all $\alpha\in \Phi$. In particular, $\Lambda_\alpha=-\Lambda_{\alpha}$ and when $\Lambda_\alpha$ is discrete, $\tilde{\Lambda}_\alpha=\Lambda_\alpha$ is isomorphic to $\Z$. 

One sets $Q^\vee= \bigoplus_{\alpha\in \Phi} \tilde{\Lambda}_\alpha \alpha^\vee$. This is a subgroup of $\A$ stable under the action of $W^v$. Then one has $W=W^v\ltimes Q^\vee$.

\subsubsection{Filters}

\begin{definition}
A filter on a set $E$ is a nonempty set $F$ of nonempty subsets of $E$ such that, for all subsets $S$, $S'$ of $E$,  if $S$, $S'\in F$ then $S\cap S'\in F$ and, if $S'\subset S$, with $S'\in F$, then $S\in F$.
\end{definition}

If $F$ is a filter on a set $E$, and $E'$ is a subset of $E$, one says that $F$ contains $E'$ if every element of $F$ contains $E'$. If $E'$ is nonempty, the set $F_{E'}$ of subsets of $E$ containing $E'$ is a filter. We will sometimes regard $E'$ as a filter on $E$ by identifying $F_{E'}$ and $E'$. If $F$ is a filter on $E$, its closure $\overline F$ (resp. its convex hull) is the filter of subsets of $E$ containing the closure (resp. the convex envelope) of some element of $F$. A filter $F$ is said to be contained in an other filter $F'$: $F\subset F'$ (resp. in a subset $Z$ in $E$: $F\subset Z$) if and only if any set in $F'$ (resp. if $Z$) is in $F$. 

If $x\in \A$ and $\Omega$ is a subset of $\A$ containing $x$ in its closure, then the \textbf{germ} of $\Omega$ in $x$ is the filter $germ_x(\Omega)$ of subsets of $\A$ containing a neighborhood of $x$ in $\Omega$.

A \textbf{sector} in $\A$ is a set of the form $\mathfrak{s}=x+C^v$ with $C^v=\pm w.C_f^v$ for some $x\in \A$ and $w\in W^v$. A point $u$ such that $\mathfrak{s}=u+C^v$ is called a \textbf{base point of} $\mathfrak{s}$ and $C^v$ is its \textbf{direction}.  The intersection of two sectors of the same direction is  a sector of the same direction.

The \textbf{sector-germ} of a sector $\mathfrak{s}=x+C^v$ is the filter $\mathfrak{S}$ of subsets of $\A$ containing an $\A$-translate of $\mathfrak{s}$. It only depends on the direction $C^v$. We denote by $+\infty$ (resp. $-\infty$) the sector-germ of $C_f^v$ (resp. of $-C_f^v$).

A ray $\delta$ with base point $x$ and containing $y\neq x$ (or the interval $]x,y]=[x,y]\backslash\{x\}$ or $[x,y]$ or the line containing $x$ and $y$) is called \textbf{preordered} if $x\leq y$ or $y\leq x$ and \textbf{generic} if $y-x\in \pm\mathring \T$, the interior of $\pm \T$.

 For $\alpha\in \Phi$, and $k\in \R\cup\{+\infty\}$, let $D(\alpha,k)=\{v\in \A| \alpha(v)+k\geq 0\}$ (and $D(\alpha,+\infty)=\A\mathrm{}$) and $D^\circ(\alpha,k)=\{v\in \A|\ \alpha(v)+k > 0\}$ (for $\alpha\in\Phi$ and $k\in \R\cup\{+\infty\}$). 

 Let $\LCC$ be the set of families $(\Lambda'_\alpha)_{\alpha\in \Phi}$ such that  $\Lambda_\alpha\subset \Lambda'_\alpha\subset \R$ and $\Lambda_{\alpha}'=-\Lambda_{-\alpha}'$, for $\alpha\in \Phi$.

An \textbf{apartment} is a root generating system equipped with a Weyl group $W$ (i.e with a set $\mathcal{M}$ of real walls, see~\ref{subsubGroupe de Weyl}) and a family $\Lambda'\in \LCC$. Let $\underline{\A}=(\mathcal{S},W,\Lambda')$ be an apartment and $\A$ be the underlying affine space.  A set of the form $M(\alpha,k)$, with $\alpha\in \Phi$ and $k\in \Lambda'_\alpha$  is called a \textbf{wall} of $\A$ and a set of the form $D(\alpha,k)$, with $\alpha\in \Phi$ and $k\in \Lambda'_\alpha$ is called a \textbf{half-apartment} of $\A$. A subset $X$ of $\A$ is said to be \textbf{enclosed} if there exist $k\in \N$, $\beta_1,\ldots,\beta_k\in \Phi$ and $(\lambda_1,\ldots,\lambda_k)\in\prod_{i=1}^k \Lambda'_{\beta_i}$ such that $X=\bigcap_{i=1}^k D(\beta_i,\lambda_i)$.

\subsection{Masure}

In this section, we define masures.

\subsubsection{Definitions of  faces, chimneys and related notions}\label{subsubDéfinition des faces,...}
Let $\underline{\A}=(\mathcal{S},W,\Lambda')$ be an apartment and $\A$ be the underlying affine space.

Let $x\in \A$ and $F^v$ be a vectorial face.  The \textbf{local-face} $F^\ell(x,F^v)=germ_x(x+F^v)$  is the filter defined as the  intersection of $x+F^v$ with the the filter of neighborhoods of $x$ in $\A$. The \textbf{face}  $F(x,F^v)$ is the filter consisting of the subsets containing a finite  intersection  of half-spaces $D(\alpha,\lambda_\alpha)$ or $D^\circ(\alpha,\lambda_\alpha)$, with $\lambda
_\alpha\in \Lambda'_\alpha\cup\{+\infty\}$ for all $\alpha\in \Phi$  (at most one $\lambda_\alpha\in \Lambda_\alpha$ for each $\alpha\in \Phi$).  We say that $F$ is \textbf{positive} (or \textbf{negative}) if $F^v$ is.

Let $F,F'$ be two faces. We say that $F'$ \textbf{dominates} $F$ if $F\subset \overline{F'}$. The dimension of a face $F$ is the smallest dimension of an affine space generated by some $S\in F$. Such an affine space is unique and is called its \textbf{support}. A face is said to be \textbf{spherical} if the direction of its support meets the open Tits cone $\mathring \T$; then its pointwise stabilizer $W_F$ in $W^v$ is finite.

A \textbf{chamber} (or alcove) is a face of the form $F(x,C^v)$ where $x\in \A$ and $C^v$ is a vectorial chamber of $\A$.

A \textbf{panel} is a face which is maximal (for the domination relation) among the faces contained in at least one wall. 

Let $F=F(x,F^v)$ be a face ($x\in \A$ and $F^v$ is a vectorial face). The \textbf{chimney}  $\mathfrak{r}(F,F^v)$ is the filter consisting of the sets containing an enclosed set containing $F+F^v$. The face $F$ is the basis of the chimney and the vectorial face $F^v$ its direction. A chimney is \textbf{splayed} if $F^v$ is spherical.

A \textbf{shortening} of a chimney $\mathfrak{r}(F,F^v)$, with $F=F(x,F_0^v)$ is a chimney of the form $\mathfrak{r} (F(x+\xi,F_0^v),F^v)$ for some $\xi\in \overline{F^v}$. The \textbf{germ} of a chimney $\mathfrak{r}$ is the filter of subsets of $\A$ containing a shortening of $\mathfrak{r}$ (this definition of shortening is slightly different from the one of \cite{rousseau2011masures} 1.12 but follows \cite{rousseau2017almost} 3.6) and we obtain the same germs with these two definitions).

\subsubsection{Masure}\label{subsubDéfinition des masures}

An \textbf{apartment of type }$\A$ is a set $\underline{A}$ with a nonempty set $\mathrm{Isom}(\A,A)$ of bijections (called \textbf{Weyl-isomorphisms}) such that if $f_0\in \mathrm{Isom}(\A,A)$ then $f\in \mathrm{Isom}(\A,A)$ if and only if there exists $w\in W$ satisfying $f=f_0\circ w$. We will say \textbf{isomorphism} instead of Weyl-isomorphism in the sequel. An isomorphism between two apartments $\phi:A\rightarrow A'$ is a bijection such that ($f\in \mathrm{Isom}(\mathbb{A},A)$ if, and only if, $\phi \circ f\in \mathrm{Isom}(\A,A')$). We extend all the notions that are preserved by $W$ to each apartment. Thus sectors, enclosures, faces and chimneys are well defined in any apartment of type $\A$.
If $A$ and $B$ are two apartments, and $\phi:A\rightarrow B$ is an  apartment isomorphism  fixing some set $X$, we write $\phi:A\overset{X}{\rightarrow} B$.

\begin{definition}\label{defMasures}
A masure of type $\underline{\A}=(\SC,W,\Lambda')$ is a set $\mathcal{I}$ endowed with a covering $\mathcal{A}$ of subsets called \textbf{apartments} such that: 

\par (MA i) Any $A\in \mathcal{A}$ is equipped with the  structure of an apartment of type $\A$.

\par (MA ii) : if two apartments $A,A'$ contain a generic ray, then $A\cap A'$ is enclosed and there exists an apartment isomorphism $\phi:A\overset{A\cap A'}{\rightarrow} A'$.

\par (MA iii): if $\RR$ is the germ of a splayed chimney and if $F$ is a face or a germ of a chimney, then there exists an apartment containing $\RR$ and $F$.
\end{definition}

In this definition, we say that an apartment contains a germ of a filter if it contains at least one element of this germ. We say that a map fixes a germ if it fixes at least one element of this germ.

The main examples of masures are the masures associated with  (almost) split Kac-Moody group over valued fields constructed in \cite{rousseau2016groupes} and \cite{rousseau2017almost}. 

By \cite[Theorem 5.1]{hebert2020new}, definition~\ref{defMasures} is equivalent to \cite[Définition 2.1]{rousseau2011masures} (at least under Assumption~\ref{assFreeness}).

\subsubsection{Example: masure associated with a split Kac-Moody group over a valued field}\label{subsecMasure associée} 
Let $A$ be a Kac-Moody matrix and $\mathcal{S}$ be a root generating system of type $A$.
We consider the group functor $\mathbf{G}$ associated with the root generating system $\mathcal{S}$ in \cite{tits1987uniqueness} and in  \cite[Chapitre 8]{remy2002groupes}. This functor is a functor from the category of rings to the category of groups satisfying axioms (KMG1) to (KMG 9) of \cite{tits1987uniqueness}. When $R$ is a field, $\mathbf{G}(R)$ is uniquely determined by these axioms by Theorem 1' of \cite{tits1987uniqueness}. This functor contains a toric functor $\mathbf{T}$, from the category of rings to the category of commutative groups (denoted $\mathcal{T}$ in \cite{remy2002groupes}) and two functors $\mathbf{U^+}$ and $\mathbf{U^-}$ from the category of rings to the category of groups. 

 Let $\mathcal{K}$ be a field equipped with a non-trivial valuation $\omega:\mathcal{K}\rightarrow \R\cup\{+\infty\}$, $\mathcal{O}$ be its ring of integers and $G=\mathbf{G}(\mathcal{K})$ (and $U^+=\mathbf{U^+}(\mathcal{K})$, ...). For all $\epsilon\in \{-,+\}$, and all $\alpha\in \Phi_\epsilon$, we have an isomorphism $x_\alpha$ from $\mathcal{K}$ to a group $U_\alpha$. For all $k\in \R$, one defines a subgroup $U_{\alpha,k}:=x_\alpha(\{u\in \mathcal{K}|\ \omega(u)\geq k\})$. Let $\I$ be the masure associated with $G$ constructed in \cite{rousseau2016groupes}. Then for all $\alpha\in \Phi$, $\Lambda_\alpha=\Lambda'_\alpha=\omega(\mathcal{K}^*)$. If moreover $\omega(\mathcal{K}^*)$ is discrete, one has (up to renormalization) $\Lambda_\alpha=\Z$ for all $\alpha\in \Phi$. 
  Moreover, we have: \begin{itemize}
\item[-] the fixator of $\A$ in $G$ is $H=\mathbf{T}(\mathcal{O})$ (by remark 3.2 of \cite{gaussent2008kac})

\item[-] for all $\alpha\in \Phi$ and $k\in \Z$, the fixator of $D(\alpha,k)$ in $G$ is $H.U_{\alpha,k}$ (by 4.2 7) of \cite{gaussent2008kac})

\item[-] for all $\epsilon\in \{-,+\}$, $H.U^\epsilon$ is the fixator of $\epsilon \infty$ (by 4.2 4) of \cite{gaussent2008kac})

\item[-] when moreover the residue field of $(\KC,\omega)$ contains $\C$,   the fixator of $\{0\}$ in $G$ is $K_s=\mathbf{G}(\mathcal{O})$ (by example 3.14 of \cite{gaussent2008kac}).

\end{itemize}

If moreover, $\mathcal{K}$ is local, with residue cardinal $q$, each panel is contained in $1+q$ chambers.

The group $G$ is reductive if and only if $W^v$ is finite. In this case, $\I$ is the usual Bruhat-Tits building of $G$ and one has $\T=\A$.

\subsubsection{Tits preorder on $\I$}
As the Tits preorder $\leq$   on $\A$ is invariant under the action of $W^v$, one can equip each apartment $A$ with $\leq_A$.  Let $A$ be an apartment of $\I$ and $x,y\in A$ be such that $x\leq_A y$. Then by   \cite[Proposition 5.4]{rousseau2011masures}, if $B$ is an apartment containing $x$ and $y$,  $[x,y]_A=[x,y]_B$ and there exists an apartment isomorphism $\psi:A\overset{[x,y]}{\rightarrow} B$. In particular, $x\leq_{B} y$. This defines a relation $\leq$  on $\I$.  By Théorème 5.9 of \cite{rousseau2011masures}, $\leq$ is a  preorder    on $\I$. It is invariant by apartment isomorphisms: if $A,B$ are apartments, $\phi:A\rightarrow B$ is an apartment isomorphism  and $x,y\in A$ are such that $x\leq y$, then $\phi(x)\leq \phi(y)$. We call it the \textbf{Tits preorder on $\I$}.

\subsubsection{Retractions centered at sector-germs}\label{subsecRétractions}
Let  $\s$ be a sector-germ of $\I$ and $A$ be an apartment containing it. Let $x\in \I$. By (MA iii), there exists an apartment $A_x$ of $\I$ containing $x$ and $\s$. By (MA ii), there exists an apartment isomorphism $\phi:A_x\rightarrow A$ fixing $\s$. By \cite[2.6]{rousseau2011masures}, $\phi(x)$ does not depend on the choices we made and thus we can set $\rho_{A,\mathfrak{s}}(x)=\phi(x)$.

The map $\rho_{A,\s}:\I\rightarrow A$ is the  \textbf{retraction onto $A$ centered at $\s$}.

We denote by $\rho_{+\infty}:\I\twoheadrightarrow \A$ (resp. $\rho_{-\infty}$) the retraction  onto $\A$ centered at $+\infty$ (resp. $-\infty$).

\section{Intersection of two apartments in a masure}\label{secIntersection_apartments}

Let $A,B$ be two apartments of $\I$. We prove below that $A\cap B$ is enclosed and that there exists $\phi:A\overset{A\cap B}{\rightarrow} B$ (see Theorem~\ref{thmIntersection_apartments}). Let us sketch our proof. We assume that $\A=B$. By results of \cite{hebert2020new}, the main difficulty is to prove that $A\cap \A$ is convex. We first assume that $A\cap \A$ has nonempty interior. Then using \cite[Proposition 9]{hebert2016distances} or \cite[Proposition 4.2.8]{hebert2018study}, we write $A=\bigcup_{i=1}^k P_i$, where $k\in \N$, the $P_i$ are enclosed and if $i\in \llbracket 1,k\rrbracket$, there exists an apartment $A_i$ containing $P_i,-\infty$ and an apartment isomorphism $A\overset{P_i}{\rightarrow} A_i$ (note that in \cite{hebert2016distances}, the masure $\I$ is moreover assumed to be semi-discrete, which means that $\Lambda'_\alpha=\Z$ for all $\alpha\in \Phi$, but this assumption is easily dropped for the mentioned results, see \cite{hebert2018study}). 
Let $\HC$ be a finite set of walls delimiting the $P_i$. If $a,b\in A$ we denote by $\tau_{a,b}:[0,1]\rightarrow A$ the affine parametrization of $[a,b]_A$ such that $\tau_{a,b}(0)=a$ and $\tau_{a,b}(1)=b$. Let $a,b\in A$ be such that at each time $t$ such that $\tau(t)$ is in some wall of $\HC$, then this wall is unique (as we shall see, almost every pair $(a,b)\in A^2$ satisfies this property). We prove that   $\pi=\rho_{-\infty}\circ \tau_{a,b}$ is a piecewise linear path whose left-hand and right-hand derivatives satisfy some growth property with respect to  $\leq_{Q^\vee_\R}$ (see Lemma~\ref{lemPaths_with_increasing_derivative}). We deduce that if $b\in A\cap \A$, then for almost all $a\in \A\cap A$, $\pi'_+(0)\leq_{Q^\vee_\R} b-a$. Applying the analogous inequality to $\rho_{+\infty}\circ \tau$, we deduce that $\rho_{-\infty}\circ \tau=\rho_{+\infty}\circ \tau$. We then deduce that $\tau\subset \A$ and we conclude by using a density argument.

\medskip
We equip every apartment with the topology defined by its structure of a finite dimensional real-affine space. If $E$ is a subset of an apartment $A$, we denote by $\In(E)$ or by $\mathring{E}$ its interior, depending on the context. This does not depend on the choice of an apartment containing $E$ by \cite[Proposition 3.26]{hebert2020new}.
Let $A$ be an apartment. For $a,b\in A$, $a\neq b$, we denote by $\LC(a,b)=\LC_A(a,b)$ the line containing $a$ and $b$ in $A$. Let  $x,y\in \I$ be such that $x\leq y$ and $A$ be an apartment containing $x$ and $y$. We write $x\mathring{<} y$ if there exists an apartment isomorphism $\phi:A\rightarrow \A$ such that $\phi(y)-\phi(x)\in \mathring{\T}$. This does not depend on the choices of $A$ and $\phi$. 

\begin{lemma}\label{lemLocal_intersection_half_apartment}
Let $x\in \I$ and $A,B$ be two apartments containing $x$. We assume that there exists a   neighborhood $V$ of $x$ in $A$ and a half-apartment $D$ of $A$ such that $A\cap B\supset D\cap V$. Then: \begin{enumerate}
\item Either $x\in \In(A\cap B)$  or there exists a neighborhood $\tilde{V}$ of $x$ in $A$ such that $A\cap B\cap \tilde{V}=D\cap \tilde{V}$.

\item If $x\notin \In (A\cap B)$, there exists an apartment $B'$ such that $B\cap B'$ is a half-apartment and such that $x\in \In (A\cap B')$. 
\end{enumerate} 
\end{lemma}

\begin{proof}
Let $M$ be the wall of $A$ delimiting $D$ and $P$ be a positive panel of $M$ based at $x$. Then $P\subset A\cap B$. Let $C$ be the chamber of $A$ dominating $P$ and  not contained in $D$. We can assume, reducing $V$ if necessary, that $V$ is convex and open. By \cite[Proposition 3.26]{hebert2020new}, there exists $\phi:A\overset{D\cap V}{\rightarrow}B$. Let $D_B=\phi(D)$. We have \begin{equation}\label{eqFirst_inclusion}
\phi(D\cap V)=D\cap V\subset \phi(D)\cap V=D_B\cap V.
\end{equation}

 Let $B'$ be an apartment  containing $D_B$ and $C$ (the existence of such an apartment is provided by \cite[Proposition 2.9 1)]{rousseau2011masures}). Then  $B'\cap A$ contains $D_B\cap V$.  Let $Q$ be the sector of $A$ based at $x$ and containing $C$. Reducing $V$ if necessary, we may assume that $V\cap Q\subset B'$. Let $y\in V\cap Q$. Then $y\mathring{>} x$. Let $V_1\subset V$ be a neighborhood of $x$ in $A$ such that $y\mathring{>} V_1$.  Let $z\in \LC(x,y)\cap \In(D\cap V)$.  Let $V_2\subset V_1$ be a neighborhood of $z$ in $A$ such that $V_2\subset D\cap V$. Then $V_3:=\bigcup_{v\in V_2}[y,v]_A$ contains $x$ in its interior.  By \cite[Proposition 5.4]{rousseau2011masures}, \begin{equation}\label{eqThird_inclusion}
 V_3\subset A\cap B'.
 \end{equation}

Suppose $B\supset C$. By applying  the result above with $B'=B$  we deduce that $x\in \In(A\cap B)$.  Consequently $A\cap B$ contains $C$ if and only if $x\in \In (A\cap B)$.  Suppose now $x\notin \In(A\cap B)$. Then $B\nsupseteq C$. Then $B'\cap B$ contains a half-apartment and thus it contains a generic ray.  Thus we have:\begin{itemize}
\item $B\cap B'$ is enclosed in $B$ (by (MA ii)),

\item $B\cap B'\nsupseteq C$,

\item $B\cap B'\supset D_B$.
\end{itemize} 

 Therefore $B\cap B'=D_B$.
 
Let $\tilde{V}\subset V_3$ be a convex neighborhood of $x$ in $A$. Let us prove the equality \begin{equation}\label{eqSecond_equation}
D\cap \tilde{V}=D_B\cap \tilde{V}.
\end{equation}

By \cite[Proposition 3.26]{hebert2020new}, there exists an apartment isomorphism $f:B'\overset{\tilde{V}}{\rightarrow} A$. By (MA ii), there exists $\psi:B'\overset{D_B}{\rightarrow} B$. By \eqref{eqFirst_inclusion}, $\phi^{-1}\circ \psi:B'\rightarrow A$  fixes pointwise $D\cap V\supset D\cap \tilde{V} $. As $D\cap \tilde{V}$ has nonempty interior, we have $f=\phi^{-1}\circ \psi$. Let now $y'\in D_B\cap \tilde{V}$. Then \[f(y')=y'=\phi^{-1}\left(\psi(y')\right)=\phi^{-1}(y')\] and thus $\phi^{-1}$ fixes  $D_B\cap \tilde{V}$ pointwise. Since $\tilde{V}\subset V$, $\phi$ also fixes pointwise $D\cap \tilde{V}$. Similarly to \eqref{eqFirst_inclusion}, we deduce that $D_B\cap \tilde{V}\subset D\cap \tilde{V}$ and $D_B\cap \tilde{V}\supset D\cap \tilde{V}$, which proves \eqref{eqSecond_equation}.

By \eqref{eqThird_inclusion}, we have $\tilde{V}\subset A\cap B'$.   Therefore: \[\begin{aligned} A\cap B\cap \tilde{V} &=A\cap B'\cap  B\cap \tilde{V} \\ &=A\cap D_B\cap \tilde{V} \\ &=A\cap D\cap \tilde{V} \text{ (by }\eqref{eqSecond_equation})\\ &=D\cap \tilde{V},\end{aligned}\] which completes the proof of the lemma.
\end{proof}

An element $r\in W^v$ is called a \textbf{reflection} if $r=wr_iw^{-1}$, for some $w\in W^v$ and $i\in I$.

Let $B$ be an apartment containing $-\infty$. We denote by $\rho_{-\infty,B}$ the retraction onto $B$ centered at $-\infty$. If $x,y\in B$ we denote $x\leq_{Q^\vee_{\R,B}} y$ if $\rho_{-\infty}(x)\leq_{Q^\vee_\R} \rho_{-\infty}(y)$.

For $t\in \R$, we denote by $]t^-,t]$ the filter on $\R$ consisting of the sets containing a set of the form $[t-\epsilon,t]$, for some $\epsilon>0$.

\begin{lemma}\label{lemWeak_version_Hecke_paths}
Let $A$ be an apartment and $x\in A$. Let $B$ be an apartment containing $-\infty$. We assume that there exists a half-apartment $D$ of $A$ and a neighborhood $V$ of $x$ in $A$ such that $V\cap A\cap B=V\cap D$. Let   $\tau:\R \rightarrow  A$ be  a non-constant affine map such that $\tau(0)=x$ and such that $\tau(]0^-,0])\subset D$. Let $\pi=\rho_{-\infty,B}\circ \tau$. Then  $\pi'_+(0)\geq_{Q^\vee_\R} \pi'_-(0)$. If moreover $\pi'_+(0)\neq \pi'_-(0)$, then   there exists a reflection $\vec{r}$ of   $W^v$ such that  $\pi'_+(0)=\vec{r}.\pi'_-(0)$. 
\end{lemma}

\begin{proof}
We assume that $\pi'_+(0)\neq \pi'_-(0)$, since otherwise there is nothing to prove. Then $x\notin \In (A\cap B)$. By Lemma~\ref{lemLocal_intersection_half_apartment}, there exists an apartment $B'$  such that $B'\cap B$ is a half-apartment and such that $B'\cap A$ contains a neighborhood of $x$.  Maybe reducing $V$, we may assume that $V$ is open and convex, that it is contained in $A\cap B'$ and that $A\cap B\cap V=D\cap V$. By \cite[Proposition 3.26]{hebert2020new}, there exists $\phi:A\rightarrow B'$ fixing $V$. Therefore there is no loss of generality in assuming that $A=B'$ and thus that $A\cap B$ contains a half-apartment. Then $A\cap B=D$ (by (MA ii)). We identify $B$ and $\A$.

Let $\psi:A\overset{A\cap \A}{\rightarrow }\A$. As $\pi'_-(0)\neq \pi'_+(0)$, $A$ does not contain $-\infty$.  Let $D_A$ be the half-apartment of $A$ opposite to $D$ and $D_{\A}$ be the half-apartment of $\A$ opposite to $D$. By \cite[Proposition 2.9 2)]{rousseau2011masures}, $\tilde{A}:=D_A\cup D_\A$ is an apartment. Let $\tilde{\psi}:\tilde{A}\overset{\tilde{A}\cap \A}\rightarrow \A$. Then as $\tilde{A}\supset -\infty$, we have $\rho_{-\infty}(y)=\tilde{\psi}(y)$ for every $y\in D_A$. Let $r$ be the reflection of the affine Weyl group $W$ with respect to the wall of $D$.  By \cite[Lemma 6 2)]{hebert2016distances}, $\psi\big(\tau(t)\big)=r.\tilde{\psi}\big(\tau(t)\big)$ for every $t\in \R_+$. Therefore $\pi'_-(0)=\vec{r}.\pi'_+(0)$, where $\vec{r}$ is the linear map associated with  $r$.  Let $\alpha\in \Phi_+$ be such that $\vec{r}=r_\alpha$. Let $\vec{D}_\A=\alpha^{-1}(\R_-)$. Then $\vec{D}_\A$ is parallel to $D_\A$. Then as $\tau(]0^-,0])\subset D$ we have $\pi'_-(0)\in \vec{D}_\A$ and thus $\alpha\big(\pi'_-(0)\big)<0$. Consequently $\pi'_+(0)=\pi'_-(0)-\alpha\big(\pi'_-(0)\big)\alpha^\vee>_{Q^\vee_\R}\pi'_-(0)$.
\end{proof}

Let $A$ be an apartment and $\q$ be a sector-germ of $\I$. Then by \cite[Proposition 9]{hebert2016distances} or \cite[Proposition 4.2.8]{hebert2018study},  there exist $k\in \N$ and enclosed subsets $P_1,\ldots,P_k$ such that:\begin{itemize}
\item $A=\bigcup_{i=1}^k P_i$,

\item for every $i\in \llbracket 1,k\rrbracket$, there exists an apartment $A_i$ containing $P_i$ and $\q$,

\item for every $i\in \llbracket 1,k\rrbracket$, there exists an apartment isomorphism $\phi_i:A\overset{P_i}{\rightarrow} A_i$,

\item $\mathring{P}_i\neq \emptyset$, for $i\in \llbracket 1,k\rrbracket$.
\end{itemize}
 
 The last condition follows from \cite[Lemma 3.10]{hebert2020new} applied with $U=X=A$.

For $i\in \llbracket 1,k\rrbracket$, we write $P_i=\bigcap_{j=1}^{k_i} D_{i,j}$, where $k_i\in \N$ and $D_{i,j}$ is a half-apartment of $A$. We denote by $M_{i,j}$ the wall of $D_{i,j}$. Then  we set $\HC(A,\q)=\{M_{i,j}| i\in \llbracket 1,k\rrbracket, j\in \llbracket 1,k_i\rrbracket\}$.  For $b\in A$, we then set \[E(A,\q,b)=A\setminus \bigcup_{M,M'\in \HC(A,\q), M\neq M'}\bigcup_{x\in M\cap M'}\LC(x,b). \]

Note that $\HC(A,\q)$ and $E(A,\q,b)$  depend on the choices we made (we can for example artificially increase the number of $P_i$). However for our purpose, this dependency will not be important and we will not be specified: for any choice, $E(A,\q,b)$ is obtained from $A$ by removing finitely many hyperplanes (which are in general not walls). It is in particular dense in $A$. By definition, we have the following lemma:

\begin{lemma}\label{lemUniqueness_wall}
Let $A$ be an apartment, $\q$ be a sector-germ of $\I$, $b\in A$  and $a\in E(A,\q,b)$. Let $\tau:[0,1]\rightarrow [a,b]_A$ be an affine parametrization of $[a,b]_A$. Then for every $t\in [0,1]$, we have either $\tau(t)\notin \bigcup_{M\in \HC(A,\q)} M$ or there exists a unique $M\in \HC(A,\q)$ such that $\tau(t)\in M$. 
\end{lemma}

A \textbf{piecewise linear continuous path of $\A$} is a continuous path $\pi:[0,1]\rightarrow \A$ such that  there exist $n\in \N$ and  $0<t_1<\ldots < t_n<1$ such that $\pi|_{[t_i,t_{i+1}]}$ is a line segment, for every $i\in \llbracket 1,n-1\rrbracket$. In \cite[6.1]{gaussent2008kac}, Gaussent and Rousseau study the image $\pi$ by $\rho_{-\infty}$ of a segment $\tau:[0,1]\rightarrow \I$ such that $\tau(0)\leq \tau(1)$. They prove that this image is a Hecke path (see \cite[Definition 5.2 and Theorem 6.2]{gaussent2008kac}). Roughly speaking, this means that $\pi$ is a piecewise linear continuous path such that for all $t$ such that $\pi'_-(t)\neq \pi'_+(t)$, $\pi'_+(t)$ ``is farther from $-\infty$'' than $\pi'_-(t)$. More precisely this means  that there exists a sequence $\xi_1=\pi'_-(t)$, $\xi_2,\ldots, \xi_n=\pi'_+(t)$  such that for all $i\in \llbracket 1,n-1\rrbracket$, $\xi_{i+1}=w_i.\xi_i$, where $w_i$ is a reflection of $W^v$ with respect to some wall $M_i$, such that if $D_i$ is the half-apartment of $\A$ delimited by $M_i$ and containing $-\infty$, then $\xi_i\in D_i$.  In the lemma below, we study the image $\pi$ by $\rho_{-\infty}$ of a segment $\tau:[0,1]\rightarrow \I$   (satisfying some technical properties). Here we do not assume that $\tau(0)\leq \tau(1)$.  We prove that $\pi$ is a piecewise linear continuous path such that for all $t$ such that $\pi'_-(t)\neq \pi'_+(t)$, $\pi'_+(t)$ ``is farther from $-\infty$'' than $\pi'_-(t)$. Thanks to our technical assumption on $\tau$, we have $n=2$ at each folding. We deduce an inequality similar to that of  \cite[1.8 Lemma]{bardy2016iwahori}.

\begin{lemma}\label{lemPaths_with_increasing_derivative}
Let $A$ be an apartment. Let $a,b\in A$. We assume that: $a\in E(A,-\infty,b)$.

Let $\tau:[0,1]\rightarrow [a,b]_A$ be the affine parametrization such that $\tau(0)=a$ and $\tau(1)=b$. Let $\pi=\rho_{-\infty}\circ \tau$. Then $\pi$ is a piecewise linear continuous path  such that:\begin{enumerate}

\item for all $t\in [0,1[$ (resp. $t\in ]0,1]$), $\pi'_+(t)\in W^v.\pi'_+(0)$ (resp. $\pi'_-(t)\in W^v.\pi'_+(0)$),

\item for every $t,t'\in [0,1]$ such that $0<t<t'<1$, we have:
 \[\pi'_+(0)\leq_{Q^\vee_\R} \pi'_-(t)\leq_{Q^\vee_\R} \pi'_+(t)\leq_{Q^\vee_\R} \pi'_-(t')\leq_{Q^\vee_\R} \pi'_+(t')\leq_{Q^\vee_\R} \pi'_-(1).\]
\end{enumerate}

We have $\rho_{-\infty}(b)-\rho_{-\infty}(a)\geq_{Q^\vee_\R}\pi'_+(0)$ and if $\pi$ is not a line segment, then this inequality is strict.  In particular, if $a,b\in  \A$, then  $\pi'_+(0)\leq_{Q^\vee_\R} b-a$.

\end{lemma}

\begin{proof}
As $\rho_{-\infty}$ is an affine morphism on every $P_i$, $i\in \llbracket 1,k\rrbracket$, $\pi$ is a piecewiese linear continuous path. Let $n\in \N$ and $0<t_1<\ldots t_n<1$ be such that $\pi|_{[t_i,t_{i+1}]}$ is a line segment, for every $i\in \llbracket 1,n-1\rrbracket$.

 We now prove that for all $i\in\llbracket 1,n\rrbracket$, $\pi'_-(t_i)\leq_{Q^\vee_\R} \pi'_+(t_i)$. Let $i\in \llbracket 1,n\rrbracket$. If $\pi'_-(t_i)=\pi'_+(t_i)$, there is nothing to prove. We assume that $\pi'_-(t_i)\neq \pi'_+(t_i)$. Let $P$ be an enclosed subset of $A$ with non-empty interior such that:\begin{enumerate}
 
 \item  $P$ is delimited by walls of $\HC:=\HC(A,-\infty)$, 
  
 \item $\tau(]t_i^-,t_i])\subset P$,
 
 \item there exist  an apartment $B$ containing $-\infty$ and $P$  and an apartment isomorphism $\phi:B\overset{P}{\rightarrow} A$.
 
 \end{enumerate}
 
 As $\pi'_-(t_i)\neq \pi'_+(t_i)$, there exists $M\in \HC$ such that  $\tau(t_i)\in M$ and by Lemma~\ref{lemUniqueness_wall}, this wall is unique. Let $V$ be a convex  neighborhood of $\tau(t_i)$ in $A$ such that $V\cap \bigcup_{M'\in \HC\setminus\{M\}} M'=\emptyset$. Let $D$ be the half-apartment of $A$ delimited by $M$ and containing $P$. Then $V\cap A\cap B\supset V\cap P$. Let us prove that $V\cap P=V\cap D$. As $D$ and $P$  are delimited by $M$, we have $P\subset D$. Suppose $V\cap D\neq V\cap P$. Let $y\in (V\cap D)\setminus P$.  Let $x\in \mathring{P}\cap V$. Then there exists $z\in [x,y]$ such that  $z$ is in the boundary $\mathrm{Fr}(P)$ of $P$. Then $z\in V$ and by assumption on $V$, we have $z\in M$. As $x\in \mathring{D}$, we have $y=z\in M$. But then $[x,y]\subset P$: a contradiction. Therefore $V\cap P=V\cap D$. Then by Lemma~\ref{lemLocal_intersection_half_apartment}, maybe replacing $V$ by  a smaller neighborhood of $\tau(t_i)$, we can assume that $V\cap A\cap B=V\cap D$. Set $\pi_{B}=\rho_{-\infty,B}\circ \tau$. By Lemma~\ref{lemWeak_version_Hecke_paths}, we have $\pi'_{B,-}(t_i)\leq_{Q^\vee_{\R,B}}\pi_{B,+}'(t_i)$ and $\pi'_-(t_i)\in W^v.\pi'_+(t_i)$. Therefore $\pi'_-(t_i)\leq_{Q^\vee_\R} \pi'_+(t_i)$. By induction we  deduce that $\pi$ satisfies (1) and  (2). By integrating $\pi'$ between $0$ and $1$, we deduce the lemma. \end{proof}

\begin{lemma}\label{lemIntersection_non-empty_interior}
Let $A,B\in \AC$ be two apartments such that $\In(A\cap B)\neq \emptyset$. Then $A\cap B$ is convex.
\end{lemma}

\begin{proof}
Using isomorphisms of apartments, we may assume that $B=\A$.  Let $a\in \In(A\cap \A)$ and $b\in A\cap \A$. We begin by proving that $[a,b]_A\subset \A$.   We choose a  sequence $(a_n)\in \In(A\cap \A)^\N$ such that $a_n\rightarrow a$ and such that for all $n\in \N$, $a_n\in E(A,-\infty,b)\cap E(A,+\infty,b)$ (this is possible since $E(A,-\infty,b)$ and $E(A,+\infty,b)$ are  obtained from $A$ by removing finitely many hyperplanes). Let $n\in \N$. Let $\tau:[0,1]\mapsto [a_n,b]_A$ be the affine parametrization of $[a_n,b]$ such that $\tau(0)=a_n$ and $\tau(1)=b$. Let $\pi=\rho_{-\infty}\circ \tau$. Then by Lemma~\ref{lemPaths_with_increasing_derivative}, $\pi'_{+}(0)\leq_{Q^\vee_\R} b-a_n$. Let $\tilde{\pi}=\rho_{+\infty}\circ \tau:[0,1]\rightarrow\A$. Then similarly to Lemma~\ref{lemPaths_with_increasing_derivative}, we have $\tilde{\pi}_{+}(0)\geq_{Q^\vee_\R} b-a_n$. Moreover $\tau([0,\epsilon)\subset \A$ for $\epsilon>0$ small enough. Therefore $\rho_{+\infty}\circ \tau|_{[0,\epsilon]}=\rho_{-\infty}\circ \tau|_{[0,\epsilon]}=\tau|_{[0,\epsilon]}$. Consequently \[\pi'_+(0)=\tilde{\pi}'_+(0)\leq_{Q^\vee_\R} b-a_n\leq_{Q^\vee_\R} \tilde{\pi}'_+(0)=\pi'_+(0).\] As $(\alpha_i^\vee)_{i\in I}$ is free, we deduce that $\pi'_+(0)=\tilde{\pi}_+(0)=b-a_n$. The path $\pi$ is thus a line segment, since otherwise  we would have $b-a_n=\pi(1)-\pi(0)>_{Q^\vee_\R} b-a_n$ by Lemma~\ref{lemPaths_with_increasing_derivative}. By symmetry, $\tilde{\pi}$ is the line segment joining $a_n$ to $b$. Therefore $\pi=\tilde{\pi}$ and thus $\rho_{-\infty}\circ\tau(t)=\rho_{+\infty}\circ \tau(t)$ for $t\in [0,1]$. By \cite[Proposition 3.7]{hebert2020new}, $\tau(t)\in \A$ for every $t\in [0,1]$. Therefore $[a_n,b]\subset \A$. Moreover \[\overline{\bigcup_{n\in \N} [a_n,b]_A}\supset [a,b]_A.\] By  \cite[Proposition 3.9]{hebert2020new}, $A\cap \A$ is closed and hence $[a,b]_A\subset \A$. 

Let $V$ be an open neighborhood of $a$ contained in $A\cap \A$. Then by what we proved, $\bigcup_{v\in V} [v,b]_{A}\subset A\cap \A$. Moreover if $n\in \N$, then $(1-\frac{1}{n})b+\frac{1}{n}V$ is a neighborhood of $(1-\frac{1}{n})b+\frac{1}{n}a$ contained in $\bigcup_{v\in V} [v,b]_{A}\subset A\cap \A$. Therefore $b\in \overline{\In(A\cap \A)}$. Therefore $\In(A\cap \A)$ is dense in $A\cap \A$. We also proved that $\In(A\cap \A)$ is convex and as $A\cap \A$ is closed, the lemma follows. 
\end{proof}

\begin{theorem}\label{thmIntersection_apartments}
Let $A$ and $B$ be two apartments. Then $A\cap B$ is enclosed and there exists an apartment isomorphism $\phi:A\rightarrow B$ fixing $A\cap B$. 
\end{theorem}

\begin{proof}
We assume that $A\cap B$ is non-empty. Let $a,b\in A\cap B$. Let $C_a$ (resp. $C_b)$ be a chamber of $A$ (resp. of $B$) based at $a$ (resp. $b$). By \cite[Proposition 5.17 (ii)]{hebert2020new}, there exists an apartment $\tilde{A}$ containing $C_a$ and $C_b$. By Lemma~\ref{lemIntersection_non-empty_interior}, $A\cap \tilde{A}$ and $\tilde{A}\cap B$ are convex.  By \cite[Proposition 3.26]{hebert2020new}, there exists $\phi_A:A\overset{A\cap \tilde{A}}{\rightarrow}\tilde{A}$ and $\phi_B:\tilde{A}\overset{\tilde{A}\cap B}{\rightarrow}B$. Set $f_{a,b}=\phi_B\circ\phi_A$. Then $[a,b]_A=[a,b]_{\tilde{A}}=[a,b]_B$ and  $f_{a,b}$ fixes $[a,b]_A$.  Therefore $A\cap B$ is convex.

 Let $H$ be the support of $A\cap B$, that is, $H$ is the smallest affine subspace of $A$ containing $A\cap B$. If $E\subset H$, we denote by $\In_H(E)$ its interior in $H$.  By \cite[Proposition 3.14]{hebert2020new}, there exist $k\in \N$, enclosed subsets $P_1,\ldots, P_k$ of $A$ such that $A\cap B=\bigcup_{i=1}^k P_i$ and such that for all $i\in \llbracket 1,k\rrbracket$, there  exists an apartment isomorphism $\phi_i:A\overset{P_i}{\rightarrow} B$.  As $A\cap B$ and the $P_i$ are convex, there exists $i\in \llbracket 1,k\rrbracket$ such that $\In_H(P_i)\neq \emptyset$. 

Let us prove that $\phi_i$ fixes $A\cap B$. Let $a\in \In_H(P_i)$. Let $b\in A\cap B$. Let $f=\phi_i^{-1}\circ f_{a,b}:A\rightarrow A$. Then $f$ fixes a neighborhood of $a$ in $[a,b]$ and as it is an affine isomorphism, it fixes $\LC(a,b)$. In particular $f(b)=b$ and hence $\phi_i(b)=f_{a,b}(b)=b$, which proves that $\phi_i$ fixes $A\cap B$. By \cite[Proposition 3.22]{hebert2020new}, $A\cap B$ is enclosed, which completes the proof of the theorem.
\end{proof}

\begin{corollary}\label{corDefinition}
Let $\I$ be a set  endowed with a covering $\mathcal{A}$ of subsets called \textbf{apartments}. We assume that $(\alpha_i)_{i\in I}$ is free in $\A^*$ and that $(\alpha_i^\vee)_{i\in I}$ is free in $\A$. Then $\I$ is a masure in the sense of \cite[Définition 2.1]{rousseau2011masures} if and only if it satisfies the following axioms:

\par (MA I)=(MA i): Any $A\in \mathcal{A}$ is equipped with the  structure of an apartment of type $\A$.

\par (MA II) : Let $A,A'$ be two apartments.  Then $A\cap A'$ is enclosed and there exists an apartment isomorphism $\phi:A\overset{A\cap A'}{\rightarrow} A'$.

\par (MA III)=(MA iii): if $\RR$ is the germ of a splayed chimney and if $F$ is a face or a germ of a chimney, then there exists an apartment containing $\RR$ and $F$.

\end{corollary}

\begin{remark}
 By \cite[Proposition 4.25 and Remark 4.26]{hebert2020new}, if $\I$ is thick (which means that each panel of $\I$ is dominated by at least three chambers)  and if there exists a group acting strongly transitively on $\I$ (see \cite[3.1.5 and Corollary 4.4.40]{hebert2018study} for the definition of such an action), then for every enclosed subset $P$ of $\A$ such that $\mathring{P}\neq \emptyset$, there exists an apartment $A$ such that $A\cap \A=P$. This is in particular true when  $\I$ is the masure associated with a split Kac-Moody group over a valued field.
\end{remark}

\bibliography{/home/auguste_pro/Documents/Projets/bibliographie.bib}

\begin{thebibliography}{BPGR16}
\bibitem[BK11]{braverman2011spherical}
Alexander Braverman and David Kazhdan.
\newblock The spherical {H}ecke algebra for affine {K}ac-{M}oody groups {I}.
\newblock {\em Annals of mathematics}, 174 (2011) pages 1603--1642.
\bibitem[BKP16]{braverman2016iwahori}
Alexander Braverman, David Kazhdan, and Manish~M. Patnaik.
\newblock Iwahori-{H}ecke algebras for {$p$}-adic loop groups.
\newblock {\em Invent. Math.}, 204(2):347--442, 2016.
\bibitem[BPGR16]{bardy2016iwahori}
Nicole Bardy-Panse, St\'ephane Gaussent, and Guy Rousseau.
\newblock Iwahori-{H}ecke algebras for {K}ac-{M}oody groups over local fields.
\newblock {\em Pacific J. Math.}, 285(1):1--61, 2016.
\bibitem[BT72]{bruhat1972groupes}
Fran{\c{c}}ois Bruhat and Jacques Tits.
\newblock Groupes r{\'e}ductifs sur un corps local.
\newblock {\em Publications Math{\'e}matiques de l'IH{\'E}S}, 41(1):5--251,
  1972.
\bibitem[BT84]{bruhat1984groupes}
Fran{\c{c}}ois Bruhat and Jacques Tits.
\newblock Groupes r{\'e}ductifs sur un corps local.
\newblock {\em Publications Math{\'e}matiques de l'IH{\'E}S}, 60(1):5--184,
  1984.
\bibitem[Cha10]{charignon2010immeubles}
Cyril Charignon.
\newblock {\em Immeubles affines et groupes de {K}ac-{M}oody}.
\newblock PhD thesis, universit{\'e} Henri Poincar{\'e} {N}ancy 1, 2010.
\bibitem[GR08]{gaussent2008kac}
St{\'e}phane Gaussent and Guy Rousseau.
\newblock Kac-{M}oody groups, hovels and {L}ittelmann paths.
\newblock In {\em Annales de l'institut Fourier}, volume~58, pages 2605--2657,
  2008.
\bibitem[GR14]{gaussent2014spherical}
St{\'e}phane Gaussent and Guy Rousseau.
\newblock Spherical {H}ecke algebras for {K}ac-{M}oody groups over local
  fields.
\newblock {\em Annals of Mathematics}, 180(3):1051--1087, 2014.
\bibitem[H{\'e}b18]{hebert2018study}
Auguste H{\'e}bert.
\newblock {Study of masures and of their applications in arithmetic. PhD thesis. English
  version}.
\newblock {\em hal.archives ouvertes tel-01856620v1}, June 2018.
\bibitem[H{\'{e}}b20]{hebert2020new}
Auguste H{\'{e}}bert.
\newblock A {N}ew {A}xiomatics for {M}asures.
\newblock {\em Canad. J. Math.}, 72(3):732--773, 2020.
\bibitem[H{\'e}b21]{hebert2016distances}
Auguste H{\'e}bert.
\newblock {Distances on a masure}.
\newblock working paper or preprint, August 2021, accepted by {\em Transformation groups}.
\newblock {https://hal.archives-ouvertes.fr/hal-01397819v3/file/distances\_on\_a\_masure.pdf}
\bibitem[R{\'e}m02]{remy2002groupes}
Bertrand R{\'e}my.
\newblock Groupes de {K}ac-{M}oody d\'eploy\'es et presque d\'eploy\'es.
\newblock {\em Ast\'erisque}, (277):viii+348, 2002.
\bibitem[Rou11]{rousseau2011masures}
Guy Rousseau.
\newblock Masures affines.
\newblock {\em Pure and Applied Mathematics Quarterly}, 7(3):859--921, 2011.
\bibitem[Rou16]{rousseau2016groupes}
Guy Rousseau.
\newblock Groupes de {K}ac-{M}oody d\'eploy\'es sur un corps local {II}.
  {M}asures ordonn\'ees.
\newblock {\em Bull. Soc. Math. France}, 144(4):613--692, 2016.
\bibitem[Rou17]{rousseau2017almost}
Guy Rousseau.
\newblock {A}lmost split {K}ac–{M}oody groups over ultrametric fields.
\newblock {\em Groups Geometry, and Dynamics}, 11:891--975, 2017.
\bibitem[Tit87]{tits1987uniqueness}
Jacques Tits.
\newblock Uniqueness and presentation of {K}ac-{M}oody groups over fields.
\newblock {\em J. Algebra}, 105(2):542--573, 1987.
\end{thebibliography}

\bibliographystyle{plain}

\end{document}